\documentclass[preprint,12pt]{elsarticle}

\makeatletter
\def\ps@pprintTitle{} 
\let\@runspecial\relax
\makeatother

\usepackage[a4paper,margin=2.5cm]{geometry}
\usepackage[T1]{fontenc}
\usepackage[utf8]{inputenc}
\usepackage{lmodern}

\usepackage{latexsym}
\usepackage{amsmath,amsfonts,amssymb,amsthm}
\usepackage{xcolor}
\usepackage{graphicx}
\usepackage[colorlinks=true,allcolors=blue,unicode]{hyperref}
\usepackage{hyphenat}

\def\FF{{\mathbb{F}}}

\def\CC{{\mathbb{C}}}
\def\NN{{\mathbb{N}}}

\def\RR{{\mathbb{R}}}

\def\ben{\begin{enumerate}}
\def\een{\end{enumerate}}
\def\l({\left(}
\def\r){\right)}

\def\cE{\mathcal{E}}
\def\cU{\mathcal{U}}
\def\cV{\mathcal{V}}

\def\cA{\mathcal{A}}
\def\cB{\mathcal{B}}

\def\cT{\mathcal{T}}
\def\cI{\mathcal{I}}
\def\cP{\mathcal{P}}

\def\cJ{\mathcal{J}}

\def\cX{\mathcal{X}}

\def\cS{\mathcal{S}}

\def\cO{\mathcal{O}}
\def\cW{\mathcal{W}}
\def\cX{\mathcal{X}}

\def\bcirc{\mathrm{bcirc}}
\def\diag{\mathrm{diag}}
\def\rank{\mathrm{rank}}

\def\fold{\mathrm{fold}}
\def\unfold{\mathrm{unfold}}

\def\ol{\overline}

\def\beqn{\begin{eqnarray*}}
\def\eeqn{\end{eqnarray*}}
\def\bequ{\begin{eqnarray}}
\def\eequ{\end{eqnarray}}
\def\barr{\begin{array}}
\def\earr{\end{array}}

\def\mxl{\left[\begin{array}}
\def\mxr{\end{array}\right]}
\def\detl{\left|\begin{array}}
\def\detr{\end{array}\right|}
\def\bmx{\begin{bmatrix}}
\def\emx{\end{bmatrix}}

\newtheorem{theorem}{Theorem}[section]
\newtheorem{lemma}[theorem]{Lemma}

\newtheorem{remark}[theorem]{Remark}
\numberwithin{equation}{section}
\def\bth{\begin{theorem}}
\def\eth{\end{theorem}}

\journal{Linear Algebra and Its Applications}
\begin{document}
\begin{frontmatter}

\title{Real tensor factorizations and generalized inverses under the $t$-product}

\author[1,2]{Faustino Maciala}
\ead{fausmacialamath@hotmail.com}

\author[3]{C. Mendes Ara\'ujo}
\ead{clmendes@math.uminho.pt}

\author[3]{Pedro Patr\'{\i}cio}
\ead{pedro@math.uminho.pt}



\affiliation[1]{organization={CMAT -- Centre of Mathematics, Universidade do Minho},
postcode={4710-057},
city={Braga},
country={Portugal}}

\affiliation[2]{organization={Departamento de Ci\^encias da Natureza e Ci\^encias Exatas do Instituto Superior de Ci\^encias da Educa\c{c}\~ao de Cabinda},
city={Cabinda},
country={Angola}}

\affiliation[3]{organization={CMAT -- Centre of Mathematics and Department of Mathematics, Universidade do Minho},
postcode={4710-057},
city={Braga},
country={Portugal}}

\begin{abstract}
The algebraic theory of third-order tensors under the $t$-product is naturally formulated over the complex field via Fourier block diagonalization. However, many applications require real-valued representations. In this paper, we investigate structural conditions ensuring that tensor factorizations and generalized inverses admit real realizations.

We show that these conditions can be characterized through the conjugate-pairing structure of the Fourier frontal slices, which determines when transform-domain constructions yield real tensors after inverse transformation. As applications, we obtain real versions of several tensor factorizations and analyze the existence and structure of associated generalized inverses.

These results provide a framework for transferring matrix-based constructions to real tensors while preserving the algebraic constraints of the $t$-product.
\end{abstract}


\end{frontmatter}

\section{Introduction}\label{sec:intro}

Third-order tensors equipped with the $t$-product provide an algebraic setting in which many classical matrix concepts extend naturally to multilinear data. A key feature of this approach is that tensor operations can be analyzed through the discrete Fourier transform along the third dimension, where they reduce to matrix operations on the corresponding frontal slices \cite{KernfeldKilmerAer2015,KilmerBramanHao2013,KilmerMartin2011}.

This Fourier-domain viewpoint underlies much of the existing algebraic and computational
theory for third-order tensors. In particular, it has been used to develop tensor singular value decompositions, tensor functions, Jordan-type constructions, and generalized inverses under the $t$-product (see, for e.g., 
\cite{ChenMaMiaoWei2023,KilmerMartin2011,Martin2013,MiaoQiWei2020,MiaoQiWei2021}).
Its appeal lies in the fact that tensor problems can often be analyzed by transporting familiar matrix constructions to the Fourier blocks and then reassembling the result through the inverse transform.

At the same time, many of the applications that motivate the $t$-product framework are inherently real-valued. This is the case, for example, in grayscale and color imaging, image deblurring, video processing, face recognition, hyperspectral imaging, tensor completion, and related inverse problems, where the underlying data are naturally modeled by real arrays and where one often seeks real reconstructions, real factorizations, or real inverse operators \cite{HaoKilmerBramanHoover2013,KernfeldKilmerAer2015,KilmerBramanHao2013,KilmerMartin2011,ReichelUgwu2021,ReichelUgwu2022,ZhangElyKilmer2014}.
From both a mathematical and computational perspective, it is therefore important to understand when transform-domain constructions actually return real tensors.

However, this issue is subtler than it may first appear. Even when the original tensor is
real-valued, its Fourier-domain representation is generally complex. As a consequence, tensor factorizations and generalized inverses constructed blockwise in the transform domain do not automatically yield real tensors after inverse transformation. What matters is not that each individual Fourier block be real, but rather that the family of blocks satisfy the conjugate symmetry relations imposed by the discrete Fourier transform.

In this paper, we study structural conditions under which tensor factorizations and generalized inverses under the $t$-product admit real realizations. The perspective adopted throughout this paper can be summarized as follows. Real-valuedness is not enforced at the level of individual Fourier blocks, but rather emerges from a coordinated choice of data across conjugate pairs in the Fourier domain. More precisely, constructions are first carried out in the block-diagonal Fourier representation, where tensor operations reduce to matrix operations, and the required algebraic properties are imposed blockwise. These choices are then coupled across conjugate indices so as to respect the symmetry relations induced by the discrete Fourier transform. Only after this compatibility is ensured do we apply the inverse transform, which then yields real tensors.

We refer to this approach as the \emph{conjugate-pairing principle}. It provides a unified mechanism for constructing real tensors from the Fourier-domain and serves as the conceptual basis for all the results developed in the paper.

Within this framework, we establish several real tensor factorizations under the $t$-product,
namely real versions of the $t$-SVD, the $t$-Schur decomposition, a Jordan-type decomposition, and a factorization through a real idempotent tensor. The common feature of these results is that real-valuedness is recovered globally through compatible choices across conjugate Fourier blocks, rather than by imposing reality independently on each block. As the examples presented later show, this distinction is essential: constructions that are formally natural over the complex field do not, in general, yield real tensor factors without an additional compatibility argument. The present approach makes this mechanism explicit and provides a consistent framework in which the resulting real factorizations are
obtained in a mathematically rigorous way.

We then apply the same principle to generalized inverses, obtaining real realizations and
transform-domain descriptions for the main generalized inverses considered in this work,
including the Moore--Penrose inverse, the Drazin inverse, and the group inverse. In addition, we show that the ring $(\mathbb{R}^{n\times n\times p},+, \ast)$ is unit regular, thereby placing these constructions in a natural ring-theoretic setting.

For background on generalized inverses over rings and, in particular, on generalized inverses of complex and real matrices, we refer the reader to \cite{BenIsraelGreville2003,CampbellMeyer2009,ChenBook}.

The paper is organized as follows. Section~\ref{sec:lemmata}  establishes the Fourier-domain framework that underlies the paper, with special attention to the conjugate-pairing relations satisfied by the Fourier blocks of real tensors and to the characterization of when inverse Fourier constructions return real tensors. Section~\ref{sec:factorizations} uses this framework to derive real tensor factorizations in the $t$-product setting by coupling matrix-level constructions across conjugate Fourier blocks. Section~\ref{sec:g-inverses} then applies the same mechanism to the study of generalized inverses, yielding real tensor versions and transform-domain representations of the corresponding inverse constructions.

\medskip
We next fix notation and recall the $t$-product framework and generalized inverses needed in the sequel.

In this paper, $\mathbb{F}$ denotes either the real field $\mathbb{R}$ or the complex field $\mathbb{C}$. We write $\mathbb{F}^{m \times n}$ for the vector space of $m \times n$ matrices over $\mathbb{F}$, and $\mathbb{F}^{m \times n \times p}$ for the vector space of third-order tensors with entries in $\mathbb{F}$.

Matrices are denoted by capital italic letters (e.g., $A \in \mathbb{F}^{m \times n}$), and tensors by calligraphic letters (e.g., $\cA \in \mathbb{F}^{m \times n \times p}$). The entry $(i,j,k)$ of
a tensor $\cA$ is denoted by $\cA_{ijk}$ .
The $k$-th frontal slice of $\cA$, $[\cA_{i,j,k}]_{i=1,\dots, m;j=1,\dots, n}$, is denoted by $A^{(k)} \in \mathbb{F}^{m \times n}$.

The transpose, the conjugate and conjugate transpose of a matrix $A$ are denoted by $A^T$, $\ol{A}$ and $A^*$, respectively. The identity matrix of order $n$ is denoted by $I_n$ and the Kronecker product of matrices $A$ and $B$ is denoted by $A \otimes B$. For matrices $A_1, \dots, A_p$, we write $\diag(A_1, \dots, A_p)$ for the block diagonal matrix with diagonal blocks $A_1, \dots, A_p$.

We now recall the basic constructions associated with the $t$-product. Let $\cA \in \FF^{m \times n \times p}$ with frontal slices $A^{(1)}, \dots, A^{(p)}$. The $\unfold$ map is defined by
\[
\unfold(\cA) =
\begin{bmatrix}
A^{(1)}\\ A^{(2)}\\ \vdots\\ A^{(p)}
\end{bmatrix}\in\FF^{(mp)\times n},
\]
and its reciprocal $\fold: \FF^{(mp) \times n} \to \FF^{m \times n \times p}$ satisfies
\[
\fold(\unfold(\cA))=\cA.
\]
The block-circulant linear operator $\bcirc: \FF^{m \times n \times p} \to \FF^{(mp) \times (np)}$ is defined by
\[
\bcirc(\cA)=
\begin{bmatrix}
A^{(1)} & A^{(p)} & \cdots & A^{(2)}\\
A^{(2)} & A^{(1)} & \cdots & A^{(3)}\\
\vdots & \vdots & \ddots & \vdots\\
A^{(p)} & A^{(p-1)} & \cdots & A^{(1)}
\end{bmatrix}.
\]
We can also define the corresponding inverse operator $\bcirc^{-1}:\FF^{(mp)\times (np)} \to \FF^{m\times n\times p}$.

Given $\cA \in \FF^{m \times n \times p}$ and $\cB \in \FF^{n \times \ell \times p}$, their \emph{$t$-product} is defined \cite{KilmerMartin2011} by
\[
\cA * \cB = \fold\big( \bcirc(\cA) \unfold(\cB) \big)
\in \FF^{m \times \ell \times p}.
\]
The sum of tensors of the same size is defined entrywise. These operations equip $\FF^{n \times n \times p}$ with a ring structure, where the identity element $\mathcal{I}$ is given by $I^{(1)} = I_n$ and $I^{(k)} = 0$ for $k \geq 2$ and the zero tensor $\cO$ is the unique tensor in $\FF^{n\times n\times p}$ whose frontal slices are all zero matrices. As usual, powers are taken with respect to $*$: $\cA^0=\cI$ and $\cA^{k}=\cA*\cA^{k-1}$ for $k\ge 1$.

Given $\alpha\in\FF$ and a tensor $\cA\in\FF^{n \times n \times p}$, we define $\alpha\cA$ as the $n\times n\times p$ tensor whose $(i,j,k)$-entry is $\alpha\cA_{i,j,k}$. The ring $\FF^{n \times n \times p}$ is, in fact, a unital algebra with unity $\cI$.

A tensor $\cA \in \FF^{n \times n \times p}$ is said to be \emph{$t$-invertible} if there exists $\mathcal{X} \in \FF^{n \times n \times p}$ satisfying $\cA * \cX = \cX * \cA = \cI$, in which case $\cX$ is called the $t$-inverse of $\cA$. 

The real tensor transpose $\cA^T$ (and, over $\CC$, the tensor conjugate transpose $\cA^H$) is obtained by taking the transpose (conjugate transpose) of each frontal slice of $\cA$ and then reversing the order of the transposed frontal slices $2$ through $p$, that is $\bcirc^{-1}\big(\bcirc(\cA)^T\big)$. 

A real tensor $\cA$ is said to be orthogonal if $\cA^{-1}=\cA^T$.

A tensor $\cA \in \mathbb R^{n\times n\times p}$ is said to be $t$-symmetric if $\cA^T = \cA$. Equivalently, $\cA$ is $t$-symmetric if its block-circulant representation $\bcirc(\cA)$ is a symmetric matrix.

A fundamental tool in the analysis of the $t$-product is the Fourier block diagonalization of block-circulant matrices. Let $\omega = e^{2\pi i/p}$ be a primitive $p$th root of unity, and set \[ \xi = \overline{\omega}=\omega^{-1}=e^{-2\pi i/p}.\] Then $\overline{\xi^m}=\xi^{-m}=\xi^{p-m}$ for every integer $m$. We consider the normalized Fourier matrix of order $p$
 \begin{equation} F_p=\frac{1}{\sqrt{p}}\big[\xi^{(j-1)(k-1)}\big]_{j,k=1}^p,\label{eq:Fp}\end{equation}
which satisfyies  $F_p^*F_p=I_p.$
 
A tensor is called \emph{$f$-diagonal}, \emph{$f$-upper triangular}, or \emph{$f$-lower triangular} if each frontal slice is, respectively, diagonal, upper triangular, or lower triangular.

Finally, we recall the main generalized inverses associated with the $t$-product. 

A \emph{von Neumann inverse} of $\cA\in \RR^{m \times n \times p}$ is any $\cX\in\RR^{n\times m\times p}$ satisfying $\cA*\cX*\cA=\cA$; a particular \emph{von Neumann inverse} of $\cA$ will be denoted by $\cA^{-}$ and the set of all such inverses by $\cA\{1\}$.

The \emph{group inverse} of $\cA\in \RR^{n \times n \times p}$, when it exists, is the unique tensor $\cA^\#\in \RR^{n \times n \times p}$ such that
\[
\cA*\cA^\#*\cA=\cA,\qquad
\cA^\#*\cA*\cA^\#=\cA^\#,\qquad
\cA*\cA^\#=\cA^\#*\cA.
\]

The \emph{Drazin inverse} of $\cA\in \RR^{n \times n \times p}$, when it exists, is the unique tensor $\cA^D\in \RR^{n \times n \times p}$ for which there exists $k\in\NN$ such that
\[
\cA^{k+1}*\cA^D=\cA^{k},\qquad
\cA^D*\cA*\cA^D=\cA^D,\qquad
\cA*\cA^D=\cA^D*\cA.
\]
The least such $k$ is the \emph{Drazin index} $i(\cA)$. If $i(\cA)=1$, then $\cA^D$ coincides with $\cA^\sharp$.

Finally, the \emph{Moore--Penrose inverse} of $\cA\in \RR^{m \times n \times p}$, when it exists, is the unique tensor $\cA^\dagger\in \RR^{n \times m\times p}$ satisfying the four Penrose equations
\[
\cA*\cA^\dagger*\cA=\cA,\qquad
\cA^\dagger*\cA*\cA^\dagger=\cA^\dagger,\qquad
(\cA*\cA^\dagger)^T=\cA*\cA^\dagger,\qquad
(\cA^\dagger*\cA)^T=\cA^\dagger*\cA.
\]

\section{Lemmata}\label{sec:lemmata}

In this section we collect a small set of technical statements that will be used repeatedly. In addition to their technical role, these results encode the conjugate symmetry constraints imposed by the discrete Fourier transform on real tensors. In particular, they identify the precise mechanism through which real-valuedness is recovered from Fourier-domain representations, namely via the pairing of conjugate blocks.

\begin{lemma}[\cite{KilmerMartin2011}]
\label{lem:bcirc-hom}
Given tensors $\cA,\cB\in\FF^{n\times n\times p}$, one has
\[
\bcirc(\cA+\cB)=\bcirc(\cA)+\bcirc(\cB),\qquad
\bcirc(\cA*\cB)=\bcirc(\cA)\,\bcirc(\cB),\qquad
\bcirc(\cI)=I_{np}.
\]
Moreover, 
\[
\bcirc:\ (\FF^{n\times n\times p},+,*)\ \longrightarrow\ (\FF^{np\times np},+,\cdot)
\]
is an algebra monomorphism. In addition, over $\RR$,
\begin{enumerate}
\item $(\cA*\cB)^T=\cB^T*\cA^T$;
\item if $\cA$ is $t$-invertible, then $\bcirc(\cA^{-1})=\bcirc(\cA)^{-1}$.
\item $\cA$ is orthogonal if and only if $\bcirc(\cA)$ is orthogonal.
\end{enumerate}
\end{lemma}

The block-circulant matrix of a tensor is unitarily similar to a block-diagonal matrix whose blocks are the Fourier transforms of its frontal slices.

\begin{lemma}[\cite{KilmerMartin2011}] 
\label{lem:bcirc-dft}
Given a tensor $\cA\in\CC^{n\times n\times p}$ with frontal slices $A^{(1)}, \dots, A^{(p)}$, we have 
\begin{equation}
\label{eq:bcirc-dft}
  \bcirc(\cA)=(F_p^\ast \otimes I_n)\,\diag(A_1,\dots,A_p) (F_p\otimes I_n).
\end{equation}
where  $A_1,\dots,A_p\in\CC^{n\times n}$ such that
\begin{equation*}
A_i=\sum_{k=1}^p \xi^{(i-1)(k-1)}A^{(k)}, \quad (i=1,\dots,p). 
\end{equation*}
are the Fourier blocks.
\end{lemma}

\begin{remark}
We remark that we may reverse the process of the previous lemma by computing
\begin{equation*}
A^{(i)}=\frac{1}{p}\sum_{k=1}^p \overline{\xi^{(i-1)(k-1)}}A_k, \quad (i=1,\dots,p). 
\end{equation*}
\end{remark}

The following result makes explicit the conjugate-pairing structure that arises when the tensor is real. Although it follows directly from the definition of the Fourier transform, we state it explicitly in the present tensor setting, as this structure is the key ingredient in all subsequent constructions.

\begin{lemma}\label{lem:conj-pairing}
Let $\mathcal A\in\RR^{n\times n\times p}$ and consider the factorization 
\[
\bcirc(\mathcal A)=(F_p^*\otimes I_n)\,\mathrm{diag}(A_1,\dots,A_p)\,(F_p\otimes I_n).
\]
Then $A_1$ is real and, for every $i\in\{2,\dots,p\}$,
\[
A_{p-i+2}=\overline{A_i}.
\]
\end{lemma}

\begin{proof}
Let $A^{(1)},\cdots,A^{(p)}$ be the frontal slices of $\cA$. Since $\mathcal A$ is real, each $A^{(k)}$ is real. Hence $A_1=\displaystyle \sum_{k=1}^p A^{(k)}$ is real.
For $i\ge2$, taking complex conjugates gives
\begin{align*}
\overline{A_i}
&=\sum_{k=1}^p \overline{\xi^{(i-1)(k-1)}}\,A^{(k)} \\
&=\sum_{k=1}^p \xi^{-(i-1)(k-1)}\,A^{(k)} \\
&=\sum_{k=1}^p \xi^{(p-(i-1))(k-1)}\,A^{(k)} \\
&=A_{p-i+2},
\end{align*}
using $\overline{\xi^m}=\xi^{-m}=\xi^{p-m}$. If $p$ is even and $i=\frac{p+2}{2}$, then $p-i+2=i$, so $A_i=\overline{A_i}$ and $A_i$ is real.
\end{proof}

Lemma \ref{lem:conj-pairing} shows that the Fourier blocks of a real tensor are not independent, but are linked through a conjugate symmetry relation. In particular, apart from the real blocks ($A_1$ and, when $p$ is even, $A_{\frac{p+2}{2}}$), all remaining blocks occur in conjugate pairs.

The next result provides the converse statement and will be used as the main tool for enforcing real-valuedness. It shows that the conjugate-pairing structure identified above is not only necessary but also sufficient for a Fourier-domain construction to correspond to a real tensor.

\begin{lemma}\label{lem:realtensor}
If $\cA\in\FF^{n\times n\times p}$ is such that $$\bcirc(\cA) = (F_p^* \otimes I_n)\, \diag (A_1,\dots,A_p) (F_p \otimes I_n),$$  where $A_1$ is real and $A_{p-k+2} = \overline{A_k}$, for all $k\in\{2,\ldots,p\}$, then $\cA$ is real.
\end{lemma} 

\proof Recall that $\overline{\xi^m}=\xi^{-m}=\xi^{p-m}$. Let $A^{(1)},A^{(2)},\ldots, A^{(p)}$ denote the frontal slices of $\cA$.
Since $$A^{(i)}=\frac{1}{p}\displaystyle \sum_{k=1}^p \overline\xi^{\ (i-1)(k-1)}A_k,$$ we obtain
\begin{align*}
\overline{A^{(i)}}&=\frac{1}{p}\displaystyle \sum_{k=1}^p \xi^{\ (i-1)(k-1)}\overline{A_k}\\
&=\frac{1}{p}A_1+\frac{1}{p}\displaystyle \sum_{k=2}^p \xi^{\ (i-1)(k-1)}A_{p-k+2}\\
&=\frac{1}{p}A_1+\frac{1}{p}\displaystyle \sum_{k=2}^p \overline{\xi^{\ (i-1)(p-k+1)}}A_{p-k+2}\\
&=\frac{1}{p}A_1+\frac{1}{p}\displaystyle \sum_{j=2}^p \overline{\xi^{\ (i-1)(j-1)}}A_j\\
&=\frac{1}{p}\displaystyle \sum_{j=1}^p \overline\xi^{\ (i-1)(j-1)}A_j\\
&=A^{(i)}.
\end{align*}
Therefore, each frontal slice $A^{(i)}$ is a real matrix, and hence the tensor $\cA$ is real. \endproof

Taken together, Lemmas \ref{lem:conj-pairing}  and \ref{lem:realtensor} show that real-valuedness of a tensor is completely determined by a conjugate-pairing condition on its Fourier blocks. This characterization will serve as the basic mechanism in all subsequent constructions: instead of requiring the Fourier blocks themselves to be real, we will select them so as to satisfy the appropriate pairing relations, which in turn ensure that the inverse transform yields a real tensor. In the next section, this principle is applied to derive real tensor factorizations from matrix-level constructions in the Fourier domain.

\section{Real tensor factorizations}\label{sec:factorizations}

We now make explicit the general principle that underlies all subsequent constructions.

\medskip

\begin{lemma}[Conjugate-pairing principle]
\label{lem:reality-transfer}
Let $\cA \in \RR^{n\times n\times p}$ and write
\[
\bcirc(\cA)=(F_p^*\otimes I_n)\diag(A_1,\dots,A_p)(F_p\otimes I_n).
\]
Assume that, for each $k=1,\dots,p$, one assigns a matrix $\Phi(A_k)$, where $\Phi$ is a matrix
operation defined on the Fourier blocks and satisfying
\[
\Phi(\overline{M})=\overline{\Phi(M)}
\]
for every matrix $M$ in its domain.

If we define a tensor $\cB$ by
\[
\bcirc(\cB)=(F_p^*\otimes I_n)\diag(\Phi(A_1),\dots,\Phi(A_p))(F_p\otimes I_n),
\]
then $\cB$ is real-valued.
\end{lemma}

\begin{proof}
Since $\cA$ is real, Lemma~\ref{lem:conj-pairing} gives
\[
A_1\in\RR^{n\times n}, \qquad A_{p-k+2}=\overline{A_k}, \quad k=2,\dots,p.
\]
Hence $\Phi(A_1)$ is real, because
\[
\Phi(A_1)=\Phi(\overline{A_1})=\overline{\Phi(A_1)}.
\]
Moreover, for each $k=2,\dots,p$,
\[
\Phi(A_{p-k+2})
=
\Phi(\overline{A_k})
=
\overline{\Phi(A_k)}.
\]
Therefore, the Fourier blocks $\Phi(A_1),\dots,\Phi(A_p)$ satisfy the conjugate-pairing conditions
of Lemma~\ref{lem:realtensor}. It follows that the tensor $\cB$ is real-valued.
\end{proof}

\medskip

In other words, real tensor structure is recovered through compatibility across conjugate
Fourier blocks, rather than through real-valuedness of the individual blocks. All subsequent
factorization results will be obtained by applying this principle to suitable matrix decompositions in the Fourier domain.

A key subtlety, which is sometimes overlooked in the literature, is that a \emph{real} tensor generally gives rise to \emph{complex} Fourier blocks that are not individually real, but are linked through conjugate pairing. As a consequence, real-valuedness is not a blockwise property, and cannot be enforced independently at the level of each Fourier slice. Any construction carried out in the Fourier domain must therefore respect this coupling in order to yield real tensor factors after the inverse transform.

This issue already arises for SVD-type constructions. In \cite[Theorem~4.1]{Martin2013}, a real $t$-SVD is proposed by factoring each Fourier block as $A_i = U_i \Sigma_i V_i^T$, with $U_i,V_i$ orthogonal and $\Sigma_i$ diagonal with nonnegative entries, and then assembling the corresponding tensor factors. However, as the next example shows, such a blockwise construction does not in general preserve real-valuedness. The obstruction is precisely that the decompositions are chosen independently, without enforcing compatibility across conjugate Fourier blocks.

Let $\cA\in\RR^{2\times 2\times 4}$ have frontal slices
\[
A^{(1)}=A^{(3)}=I_2,\qquad
A^{(2)}=\begin{bmatrix}-1&0\\0&0\end{bmatrix},\qquad
A^{(4)}=\begin{bmatrix}0&0\\0&1\end{bmatrix},
\]
and let $F_4$ be the Fourier matrix, defined as in \eqref{eq:Fp},
\[
F_4=\frac12\begin{bmatrix}
1&1&1&1\\
1&-i&-1&i\\
1&-1&1&-1\\
1&i&-1&-i
\end{bmatrix}.
\]
Under the Fourier block diagonalization
\[
\bcirc(\cA)=(F_4^*\otimes I_2)\,\diag(A_1,\ldots,A_4)\,(F_4\otimes I_2),
\]
one obtains, in particular,
\[
A_2=A^{(1)}-iA^{(2)}-A^{(3)}+iA^{(4)}=\begin{bmatrix}i&0\\0&i\end{bmatrix}.
\]
Suppose, for the sake of argument, that $A_2$ admits a factorization of the form
$A_2=U\Sigma V^T$, where $U$ and $V$ are real orthogonal matrices and $\Sigma$ is diagonal with nonnegative entries. Then
\[
A_2A_2^T=U\Sigma^2U^T,
\]
which implies that $A_2 A_2^T$ is real symmetric and positive semidefinite. However, a direct computation shows that $A_2 A_2^T = \diag(-1,-1)$, which is negative definite. This contradiction shows that such a factorization cannot exist.

This example shows that, even for real tensors, it is not possible in general to select real orthogonal SVD factors independently for each Fourier block. The correct
construction is instead obtained by coupling singular value decompositions across conjugate pairs of blocks. 

The next theorem establishes the same real $t$-SVD statement as in \cite[Theorem~4.1]{Martin2013}, but with a proof that makes this compatibility explicit.

\begin{theorem}[real $t$-SVD]\label{thm:realtSVD}
Let $\cA\in \mathbb{R}^{n \times n \times p}$. Then there exist orthogonal tensors $\cU, \cV\in \mathbb{R}^{n \times n \times p}$ and an $f$-diagonal tensor $\cS\in \mathbb{R}^{n \times n \times p}$  such that
\[
\cA=\cU\ast\cS\ast\cV^T .
\]
\end{theorem}
\proof Let $\cA$ be a real tensor of order $n\times n\times p$ and consider its Fourier block diagonalization
\begin{equation*}
\bcirc(\cA) = (F_p^* \otimes I_n) \diag(A_1,\dots,A_p) (F_p \otimes I_n).\end{equation*}

By Lemma \ref{lem:conj-pairing}, the blocks satisfy the conjugate-pairing relations: $A_1$ is real and $A_{p-k+2}=\overline{A_k}$ for all $k\in\{2,\ldots,p\}$. Observe that $p-k+2=k$ if and only if $k=\frac{p+2}{2}$. Consequently, if $p$ is even, exactly two diagonal blocks are real, namely $A_1$ and $A_{\frac{p+2}{2}}$. If $p$ is odd, the only real block is $A_1$.

We construct singular value decompositions of the Fourier blocks in a way that is compatible with this pairing.

Let $N=\{k\in\NN\mid 2\le k<\tfrac{p+2}{2}\}$. For each $k\in N$, choose a singular value decomposition $A_k = U_k \Sigma_k V_k^\ast$, where $U_k$ and $V_k$ are unitary and $\Sigma_k$ is diagonal with nonnegative entries. We then define
\[
U_{p-k+2} = \overline{U_k}, \quad 
\Sigma_{p-k+2} = \Sigma_k, \quad 
V_{p-k+2} = \overline{V_k},
\]
which yields compatible singular value decompositions for the conjugate blocks. Indeed, since
$$A_{p-k+2}=\overline{A_k}=\overline{U_k}\, \Sigma_k\,  \overline{V_k}^\ast,$$
the required relation follows immediately. This construction is an explicit implementation of the conjugate-pairing principle.

For the real blocks, we proceed separately. The matrix $A_1$ admits a real singular value decomposition
\[
A_1 = U_1 \Sigma_1 V_1^T,
\]
with $U_1$ and $V_1$ orthogonal. If $p$ is even, the same holds for $A_{\frac{p+2}{2}}$.

Define the block diagonal matrices
\[
U = \mathrm{diag}(U_1,\ldots,U_p), \quad
\Sigma = \mathrm{diag}(\Sigma_1,\ldots,\Sigma_p), \quad
V = \mathrm{diag}(V_1,\ldots,V_p).
\]
By construction, $\mathrm{diag}(A_1,\ldots,A_p) = U \Sigma V^\ast$.

Let $\cU$, $\cS$ and $\cV$ be the tensors defined by $\bcirc(\cU) = (F_p^* \otimes I_n) U (F_p \otimes I_n)$, $\bcirc(\cS) = (F_p^* \otimes I_n) \Sigma (F_p \otimes I_n)$ and $\bcirc(\cV) = (F_p^* \otimes I_n) V  (F_p \otimes I_n)$, respectively.

By construction, the block data satisfy the conjugate-pairing conditions of Lemma \ref{lem:realtensor}. It follows that $\cU$, $\cS$, and $\cV$ are real-valued tensors. Moreover, since each $\Sigma_k$ is diagonal, $\cS$ is $f$-diagonal.

Since each $U_k$ and $V_k$ is unitary, the block diagonal matrices $U$ and $V$ are unitary. As $F_p$ is unitary, the matrix $(F_p \otimes I_n)$ is also unitary, and $(F_p \otimes I_n)^{-1}=(F_p^* \otimes I_n)$. We compute
\begin{align*}
\bcirc(\cU)\bcirc(\cU)^\ast
&=
(F_p^* \otimes I_n)U(F_p \otimes I_n)
(F_p^* \otimes I_n)U^\ast(F_p \otimes I_n)\\
&=
(F_p^* \otimes I_n)UU^\ast(F_p \otimes I_n)\\
&=I_{np}.
\end{align*}
Thus, $\bcirc(\cU)^\ast=\bcirc(\cU)^{-1}$ and $\bcirc(\cU)$ is a unitary matrix. Since $\cU$ is real-valued, this implies $\cU^T\ast\cU=\cI$ and therefore $\cU$ is an orthogonal tensor. The same argument shows that $\cV$ is also an orthogonal tensor.

Finally, from $A=U\Sigma V^\ast$, applying the inverse Fourier similarity yields
\[
(F_p^*\otimes I_n)A(F_p\otimes I_n)
=
(F_p^*\otimes I_n)U(F_p\otimes I_n)
(F_p^*\otimes I_n)\Sigma(F_p\otimes I_n)
(F_p^*\otimes I_n)V^\ast(F_p\otimes I_n).
\]
By the definitions of $\cU$, $\cS$, and $\cV$, this is precisely
$$\bcirc(\cA)=\bcirc(\cU)\bcirc(\cS)\bcirc(\cV)^\ast.$$
Since $\cV$ is real-valued, we have $\bcirc(\cV)^\ast=\bcirc(\cV^T)$, and therefore
 \begin{align*}
\bcirc(\cA)&=\bcirc(\cU)\bcirc(\cS)\bcirc(\cV^T)\\
&=\bcirc(\cU\ast\cS\ast\cV^T),
\end{align*}
where the last equality follows from Lemma \ref{lem:bcirc-hom}. As \(\bcirc\) is injective, it follows that $\cA=\mathcal U*\cS*\mathcal V^T$.
\endproof

\begin{remark}\label{rem:realtSVD} Theorem \ref{thm:realtSVD} extends verbatim to tensors
$\cA \in \mathbb{R}^{m\times n\times p}$. In that case, one obtains orthogonal tensors
$\cU \in \mathbb{R}^{m\times m\times p}$ and
$\cV \in \mathbb{R}^{n\times n\times p}$, together with an
$f$-diagonal tensor $\cS \in \mathbb{R}^{m\times n\times p}$, such that
$\cA = \cU \ast \cS \ast \cV^T$.
\end{remark}
\medskip

We next show that the same strategy yields a real Schur-type decomposition in the $t$-product setting (cf. \cite{ChenMaMiaoWei2023}).

 \begin{remark}
Recall first the classical Schur decomposition for matrices. Let $A\in\CC^{n\times n}$ be a square complex matrix. Then there exist a unitary matrix $U\in\CC^{n\times n}$ and an upper triangular matrix $T\in\CC^{n\times n}$ such that
\[
A = U T U^{*}
\]
and such that the diagonal entries of $T$ are precisely the eigenvalues of $A$.

In general, the matrices $U$ and $T$ are complex, even when $A$ is real. In order to remain in the real field one considers the real Schur decomposition. In this case, for a matrix $A\in\RR^{n\times n}$ there exist an orthogonal matrix $U\in\RR^{n\times n}$ and a quasi-triangular matrix $T\in\RR^{n\times n}$ such that
\[
A = U T U^{T}
\]
and such that the diagonal blocks of $T$ are of size $1\times 1$ and $2\times 2$. The $1\times 1$ blocks are the real eigenvalues of $A$, while each $2\times 2$ block, of the form $\bmx a&b\\-b&a\emx$ corresponds to a pair of complex conjugate eigenvalues $a\pm ib$ of $A$. Moreover, the diagonal blocks of $T$ can be arranged in an arbitrary order. In particular, one can choose a real Schur decomposition in which all $1\times 1$ blocks appear first, followed by the $2\times 2$ blocks.
\end{remark}

The real Schur decomposition motivates the tensor analogue in the framework of the $t$-product. Indeed, when working with real tensors, the block diagonalization obtained via the Fourier transform generally produces complex matrices that appear in conjugate pairs. As in the matrix case, it is therefore desirable to obtain a decomposition that remains entirely in the real field.

This leads to the notion of an \emph{$f$-quasi-triangular tensor}. A tensor $\cT\in\mathbb R^{n\times n\times p}$ is said to be $f$-quasi-triangular if each frontal slice $T^{(i)}$ is a quasi-triangular matrix, that is an upper block triangular matrix with diagonal blocks of size at most $2\times2$. Equivalently, there exists a partition $n=n_1+\cdots+n_m$ with $n_j\in\{1,2\}$ such that every $T^{(i)}$ is upper block triangular with respect to this fixed partition.

The following theorem shows that every real tensor admits such a decomposition through orthogonal similarity in the $t$-product sense, providing the tensor analogue of the real Schur decomposition.

\begin{theorem}[real $t$-Schur decomposition]
Let $\cA\in \mathbb{R}^{n \times n \times p}$. Then there exist an orthogonal tensor $\cU\in \mathbb{R}^{n \times n \times p}$ and an $f$-quasi-triangular tensor $\cT\in \mathbb{R}^{n \times n \times p}$ such that
\[
\cA=\cU\ast\cT\ast\cU^T .
\]
\end{theorem}

\proof Let $\cA$ be a real tensor of order $n\times n\times p$ and consider its Fourier block diagonalization
\begin{equation*}
\bcirc(\cA) = (F_p^* \otimes I_n) \diag(A_1,\dots,A_p) (F_p \otimes I_n).\end{equation*}

By Lemma \ref{lem:conj-pairing}, the Fourier blocks satisfy the conjugate-pairing relations: $A_1$ is real and $A_{p-k+2}=\overline{A_k}$ for all $k\in\{2,\dots,p\}$.

We construct Schur decompositions of the Fourier blocks in a way that is compatible with this pairing.

For the real block $A_1$, we take a real Schur decomposition
\[
A_1 = U_1 T_1 U_1^T,
\]
where $U_1$ is orthogonal and $T_1$ is quasi-triangular, with diagonal blocks of size at most $2 \times 2$, chosen so that all $1\times1$ blocks appear first, followed by the $2\times2$ blocks. If $p$ is even, the same construction applies to $A_{\frac{p+2}{2}}$.

 Let $N = \{\, k \in \mathbb{N} : 2 \le k < \tfrac{p+2}{2} \,\}$. For each $k \in N$, choose a (complex) Schur decomposition
\[
A_k = U_k T_k U_k^*,
\]
where $U_k$ is unitary and $T_k$ is upper triangular. We then define
\[
U_{p-k+2} = \overline{U_k}, \quad
T_{p-k+2} = \overline{T_k},
\]
which yields compatible Schur decompositions of the conjugate blocks.

This construction provides an explicit implementation of the conjugate-pairing principle.

Define the block diagonal matrices
$$U =
\diag(U_1,\dots,U_p),
\qquad
T =\diag(T_1,\dots,T_p).$$
By construction, $A_k = U_k T_k U_k^{*}$, for all $k=1,\dots,p$, and therefore
$$
\diag(A_1,\dots,A_p)
=
U T U^{*}.$$

We now define tensors $\cU$ and $\cT$ by
$$\mathrm{bcirc}(\cU)=
(F_p^{*}\otimes I_n)U(F_p\otimes I_n),
\qquad
\mathrm{bcirc}(\cT)
=
(F_p^{*}\otimes I_n)T(F_p\otimes I_n).$$

By construction, the block data satisfy the conjugate-pairing conditions of Lemma \ref{lem:realtensor}, and therefore $\cU$ and $\cT$ are real-valued tensors. 

Moreover, each $U_k$ is unitary, hence $U$ is unitary and therefore
$\bcirc(\cU)$ is unitary. Since $\cU$ is real-valued, it follows that
$\cU$ is an orthogonal tensor. 

Recall that the frontal slices $T^{(1)},\ldots, T^{(p)}$ of $\cT$ are obtained from the inverse Fourier transform,
namely $$T^{(i)}
=
\frac{1}{p}
\displaystyle\sum_{k=1}^{p}
\overline{\xi}^{(i-1)(k-1)} T_k,
\qquad i=1,\dots,p.$$

We begin by considering the case where $p$ is odd. By construction, $T_1$ is quasi-triangular, while $T_k$ is upper triangular for all $k\ge 2$. Since a linear combination of upper triangular matrices is again upper triangular, the matrices $$\displaystyle\sum_{k=2}^{p}\overline{\xi}^{(i-1)(k-1)}T_k$$
are upper triangular. Consequently, each frontal slice $T^{(i)}$ is obtained as the sum of a
quasi-triangular matrix and an upper triangular matrix. Observe that the $2\times2$ diagonal blocks of $T_1$ are preserved under scalar
multiplication and addition with upper triangular matrices. Hence each frontal slice $T^{(i)}$ is again
quasi-triangular with the same diagonal block structure as $T_1$.

Now suppose that $p$ is even. In this case, both $T_1$ and $T_{\frac{p+2}{2}}$ are quasi-triangular, with all $1\times1$ diagonal blocks first, followed by the $2\times2$ diagonal blocks. It follows that any linear combination of these two matrices is again quasi-triangular. Moreover, for $k\neq 1,\frac{p+2}{2}$, the matrices $T_k$ are upper triangular, and therefore any linear combination of such matrices is upper triangular. Hence each frontal slice $T^{(i)}$ can be written as the sum of a quasi-triangular matrix and an upper triangular matrix. This shows that each $T^{(i)}$ is again quasi-triangular.

Therefore, in both cases, the tensor $\mathcal T$ has quasi-triangular frontal slices, which means $\cT$ is $f$-quasi-triangular.

Finally, from
$$
\bcirc(\cA)=\bcirc(\cU)\,\bcirc(\cT)\,\bcirc(\cU)^{*}$$
and the injectivity of $\bcirc$, we conclude that $A = \cU\ast \cT \ast \cU^{T}.$
\endproof

The preceding argument highlights once more that real tensor structure is recovered globally through conjugate pairing in the Fourier domain.

\medskip

This Fourier-domain viewpoint is also the basis of the $t$-Jordan construction introduced in
\cite[Theorem~1]{MiaoQiWei2021}, where it is shown that any tensor
$\cA\in\mathbb{C}^{n\times n\times p}$ admits a factorization
\[
\cA=\mathcal{P}\ast\mathcal{J}\ast\mathcal{P}^{-1},
\]
with $\mathcal{P}$ t-invertible and $\mathcal{J}$ an $f$-upper-bi-diagonal tensor whose Fourier blocks are in Jordan canonical form. 

Immediately after this theorem, it is asserted that if a tensor $\cA \in \mathbb{R}^{n \times n \times p}$ is real, then its $t$-Jordan canonical form $\cJ$ is also real.

We now revisit this statement and show that, in general, such a construction need not preserve real-valuedness. The underlying issue is that the Jordan forms of the Fourier blocks are selected independently, without enforcing compatibility across conjugate pairs. This is illustrated by the following example.

Let $n=2$ and $p=4$, and consider the real tensor $\cA \in \mathbb{R}^{2 \times 2 \times 4}$ whose frontal slices are given by
\[
A^{(1)} =
\begin{pmatrix}
1 & -1\\
0 & -1
\end{pmatrix}, \quad
A^{(2)} =
\begin{pmatrix}
-1 & 0\\
-1 & 0
\end{pmatrix},
\]
\[
A^{(3)} =
\begin{pmatrix}
0 & -1\\
1 & 1
\end{pmatrix}, \quad
A^{(4)} =
\begin{pmatrix}
0 & 1\\
1 & 0
\end{pmatrix}.
\]

Let $F_4$ be the Fourier matrix, defined as in \eqref{eq:Fp}. Following the construction in the proof of Theorem 1 in \cite{MiaoQiWei2021}, we compute the matrices $B^{(k)}$ obtained via the Fourier block diagonalization of $\mathrm{bcirc}(\cA)$. A direct calculation yields
\[
B^{(1)} =
\begin{pmatrix}
0 & -1\\
1 & 0
\end{pmatrix}, \quad
B^{(2)} =
\begin{pmatrix}
1+i & i\\
-1+2i & -2
\end{pmatrix},
\]
\[
B^{(3)} =
\begin{pmatrix}
2 & -3\\
1 & 0
\end{pmatrix}, \quad
B^{(4)} =
\begin{pmatrix}
1-i & -i\\
-1-2i & -2
\end{pmatrix}.
\]

We now compute the Jordan canonical form of each $B^{(k)}$. Since all matrices are diagonalizable, we obtain
\[
C^{(1)} =
\begin{pmatrix}
-i & 0\\
0 & i
\end{pmatrix}, \quad
C^{(2)} =
\begin{pmatrix}
-1 & 0\\
0 & i
\end{pmatrix},
\]
\[
C^{(3)} =
\begin{pmatrix}
1+\sqrt{2}i & 0\\
0 & 1-\sqrt{2}i
\end{pmatrix}, \quad
C^{(4)} =
\begin{pmatrix}
-1 & 0\\
0 & -i
\end{pmatrix}.
\]

According to equation (3) in \cite{MiaoQiWei2021}, the frontal slice $J^{(1)} $ of the tensor $\cJ$ is given by
\[
J^{(1)} = \frac{1}{4}\left(C^{(1)} + C^{(2)} + C^{(3)} + C^{(4)}\right).
\]

Substituting the expressions above, we obtain
\[
J^{(1)} =
\frac{1}{4}
\begin{pmatrix}
-1 + (\sqrt{2}-1)i & 0\\
0 & 1 + (1-\sqrt{2})i
\end{pmatrix},
\]
which has nonzero imaginary parts, and therefore is not real. Consequently, the tensor $\cJ$ obtained through this procedure is not real.

This example shows that, although the original tensor $\cA$ is real, the construction described in \cite{MiaoQiWei2021} does not, in general, preserve real-valuedness. The obstruction again lies in the lack of compatibility between the Jordan data associated with conjugate Fourier blocks. In order to remain within the real setting, one must therefore impose a conjugate-pairing condition on the Fourier-domain Jordan construction. This is analogous in spirit to the passage from the complex Jordan canonical form to the real Jordan canonical form in matrix theory, which we recall next for context.

\begin{remark}
Recall first the classical Jordan canonical form for matrices over $\mathbb{C}$. Given $A \in \mathbb{C}^{n \times n}$, there exists an invertible matrix $P \in \mathbb{C}^{n \times n}$ such that
\[
A = PJP^{-1},
\]
where $J$ is a block diagonal matrix, called the Jordan canonical form of $A$, consisting of Jordan blocks
\[
J = \diag\big(J_{m_1}(\lambda_1), \ldots, J_{m_k}(\lambda_k)\big),
\]
with each block of the form
\[
J_m(\lambda) =
\bmx
\lambda & 1 &        & 0 \\
        & \lambda & \ddots &   \\
        &         & \ddots & 1 \\
0       &         &        & \lambda
\emx \in \mathbb{C}^{m \times m}.
\]
The scalars $\lambda_i \in \mathbb{C}$ are the eigenvalues of $A$, and the sizes of the Jordan blocks are determined by the algebraic and geometric multiplicities of these eigenvalues. The Jordan canonical form is unique up to permutation of the blocks.

When $A$ is real, the matrices $P$ and $J$ above need not be real. To remain over $\mathbb{R}$, one considers the real Jordan canonical form. In this case, for $A \in \mathbb{R}^{n \times n}$, there exists an invertible matrix $P \in \mathbb{R}^{n \times n}$ such that
\[
A = PJP^{-1},
\]
where $J$ is a block diagonal matrix whose diagonal blocks are of two types:
\begin{itemize}
    \item real Jordan blocks corresponding to real eigenvalues $\lambda \in \mathbb{R}$,
    \[
    J_m(\lambda) =
    \bmx
    \lambda & 1 &        & 0 \\
            & \lambda & \ddots &   \\
            &         & \ddots & 1 \\
    0       &         &        & \lambda
    \emx,
    \]
    \item real blocks associated with complex conjugate eigenvalue pairs $a \pm ib$ $(b \neq 0)$, given by
    \[
    C_m(a,b) =
    \bmx
    D & I_2 &        & 0 \\
      & D   & \ddots &   \\
      &     & \ddots & I_2 \\
    0 &     &        & D
    \emx, \quad
    D =
    \bmx
    a & b \\
    -b & a
    \emx.
    \]
\end{itemize}
Since the real Jordan form is unique only up to permutation of its diagonal blocks, one may order them conveniently. In particular, one may assume that all blocks associated with real eigenvalues appear first, followed by those associated with complex conjugate pairs. This does not affect the canonical nature of the decomposition. Thus, the real Jordan canonical form provides a canonical representation entirely in real matrices, while encoding nonreal eigenvalues through $2\times2$ real blocks.
\end{remark}

Motivated by the preceding discussion, we now show how the conjugate-pairing principle yields a real $t$-Jordan canonical form.

We now formalize the block structure that appears in the real tensor Jordan-type decomposition.

A matrix $J \in \RR^{n\times n}$ is said to be \emph{upper block-bi-diagonal} if there exists a partition $n = n_1 + \dots + n_r$ ($n_j \in \{1,2\}$) such that $J$ is block upper triangular and only the diagonal blocks $J_{ii}$, of size $n_i \times n_i$, and the first superdiagonal blocks $J_{i,i+1}$, of size $n_i \times n_{i+1}$, may be nonzero. A tensor $\cJ$ is said to be \emph{$f$-upper-block-bi-diagonal} if each frontal slice is upper block-bi-diagonal.

\begin{theorem}[real $t$-Jordan canonical form]\label{thm:real-tjordan}
Let $\cA\in\RR^{n\times n\times p}$. Then there exist a $t$-invertible tensor
$\cP\in\RR^{n\times n\times p}$ and an $f$-upper-block-bi-diagonal tensor $\cJ\in\RR^{n\times n\times p}$ such that
$$\cA=\cP*\cJ*\cP^{-1}.$$
Moreover, each frontal slice of $\cJ$ is upper block-bi-diagonal, with diagonal and first superdiagonal blocks of size at most $2\times2$.
\end{theorem}

\begin{proof}
Consider the factorization \begin{equation*}
\bcirc(\cA) = (F_p^* \otimes I_n) \diag(A_1,\ldots,A_p) (F_p \otimes I_n)\end{equation*}
given in Lemma \ref{lem:bcirc-dft}.

By Lemma \ref{lem:conj-pairing}, $A_1$ is real and $A_{p-k+2}=\overline{A_k}$ for all $k\in\{2,\dots,p\}$.

Set $N=\{\,k\in\NN: 2\le k<\tfrac{p+2}{2}\,\}$. For each $k\in N$, choose a Jordan decomposition
\[
A_k=P_k\,J_k\,P_k^{-1},
\]
with $J_k$ in Jordan canonical form over $\CC$. Then
\[
A_{p-k+2}=\overline{A_k}=\overline{P_k}\,\overline{J_k}\,(\overline{P_k})^{-1}
\]
is a Jordan decomposition of $A_{p-k+2}$. Let $P_{p-k+2}=\overline{P_k}$ and $J_{p-k+2}=\overline{J_k}$. This construction provides an explicit implementation of the conjugate-pairing principle.

For the real block(s), namely $k=1$ and, when $p$ is even, also $k=\frac{p+2}{2}$, choose real Jordan decompositions
\[
A_k=P_kJ_kP_k^{-1},
\qquad
P_k,J_k\in\RR^{n\times n},
\]
with $J_k$ in real Jordan canonical form, ordered so that the blocks associated with real eigenvalues precede those associated with complex conjugate pairs.

Define
\[
P=\diag(P_1,\ldots,P_p),
\qquad
J=\diag(J_1,\ldots,J_p).
\]
Then
\[
\diag(A_1,\ldots,A_p)=PJP^{-1}.
\]
Now define tensors $\cP,\cJ\in\CC^{n\times n\times p}$ by
\[
\bcirc(\cP)=(F_p^*\otimes I_n)\,P\,(F_p\otimes I_n),
\qquad
\bcirc(\cJ)=(F_p^*\otimes I_n)\,J\,(F_p\otimes I_n).
\]
By construction, the block data satisfy the conjugate-pairing conditions of Lemma \ref{lem:conj-pairing}, and therefore $\cP$ and $\cJ$ are real-valued tensors.

We now show that $\cJ$ is $f$-upper-block-bi-diagonal. Its frontal slices satisfy
\[
J^{(i)}=\frac1p\sum_{k=1}^{p}\overline{\xi}^{(i-1)(k-1)}J_k,
\qquad i=1,\dots,p.
\]

Assume first that $p$ is odd. Then $J_1$ is in real Jordan canonical form and is therefore upper block-bi-diagonal, with diagonal and first superdiagonal blocks of size at most $2\times2$. More precisely, $J_1$ is either upper bi-diagonal or a direct sum of an upper bi-diagonal matrix and an upper block-bi-diagonal matrix with $2\times2$ diagonal and first superdiagonal blocks. For every $k\ge2$, the matrix $J_k$ is in Jordan canonical form over $\CC$ and hence upper bi-diagonal. Therefore, $\displaystyle\sum_{k=2}^p \alpha_k J_k$ is upper bi-diagonal for any scalars $\alpha_k$, and adding such a matrix to $\alpha_1 J_1$ cannot create nonzero entries outside the diagonal and first superdiagonal block positions already present in $J_1$, for any scalar $\alpha_1$. It follows that each frontal slice $J^{(i)}$ is upper block-bi-diagonal, with diagonal and first superdiagonal blocks of size at most $2\times2$.

Now assume that $p$ is even. In this case, both $J_1$ and $J_{\frac{p+2}{2}}$ are in real Jordan canonical form. Hence each is either upper bi-diagonal or a direct sum of an upper bi-diagonal matrix and an upper block-bi-diagonal matrix with $2\times2$ diagonal and first superdiagonal blocks. Consequently, any linear combination of $J_1$ and $J_{\frac{p+2}{2}}$ is again upper block-bi-diagonal, with diagonal and first superdiagonal blocks of size at most $2\times2$. Moreover, for $k\neq 1,\frac{p+2}{2}$, each $J_k$ is upper bi-diagonal, and so any linear combination of such matrices is again upper bi-diagonal. Therefore, each frontal slice $J^{(i)}$ is upper block-bi-diagonal, with diagonal and first superdiagonal blocks of size at most $2\times2$.

As $\bcirc(\cP)$ is invertible, the tensor $\cP$ is $t$-invertible. Moreover, Lemma \ref{lem:bcirc-dft} gives
\begin{align*}
\bcirc(\cA)
&=(F_p^*\otimes I_n)\,\diag(A_1,\ldots,A_p)\,(F_p\otimes I_n) \\
&=\bcirc(\cP)\,\bcirc(\cJ)\,\bcirc(\cP)^{-1} \\
&=\bcirc(\cP*\cJ*\cP^{-1}).
\end{align*}
The injectivity of \(\bcirc\) therefore yields
\[
\cA=\cP*\cJ*\cP^{-1}.
\]
\end{proof}

\begin{remark}
The block structure obtained in Theorem \ref{thm:real-tjordan} is consistent with that of the real Jordan canonical form for matrices, in the sense that only blocks of size $1\times 1$ and $2\times 2$ appear. 

This analogy, however, is purely structural. In the matrix case, the real Jordan canonical form is a spectral decomposition, with blocks corresponding directly to real eigenvalues and complex conjugate pairs. In contrast, the tensor construction arises from a conjugate-pairing of Fourier-domain Jordan data, and not from a direct spectral decomposition of $\cA$.

Accordingly, the frontal slices of $\cJ$ are linear combinations of Jordan-form matrices arising in the Fourier domain, and therefore do not in general retain a Jordan structure. Thus, while the construction preserves a real block structure of size at most $2\times 2$, it
should not be interpreted as a spectral decomposition in the classical matrix sense.
\end{remark}

\medskip
We conclude the section with a factorization through a real idempotent tensor. Once again, the
construction is carried out at the level of the Fourier blocks and then coupled across conjugate
pairs so as to ensure that the resulting tensors remain real-valued.

\begin{theorem}[real idempotent tensor]\label{thm:rank-factor-idempotent}
Let $\cA \in \RR^{n\times n\times p}$. Then there exist invertible real tensors
$\cU,\cV\in\RR^{n\times n\times p}$ and a real idempotent tensor $\cE\in\RR^{n\times n\times p}$ such that
\[
\cA=\cU*\cE*\cV.
\]
\end{theorem}

\begin{proof}
Let $\cA \in \RR^{n\times n\times p}$ and consider its Fourier block diagonalization
\begin{equation}\label{eq:fft-blockdiag-A-idem}
\bcirc(\cA)=(F_p^*\otimes I_n)\diag(A_1,\dots,A_p)(F_p\otimes I_n).
\end{equation}
By Lemma~\ref{lem:conj-pairing}, the Fourier blocks satisfy the conjugate-pairing relations: $A_1$ is real and $A_{p-k+2}=\overline{A_k}$ for all $k \in \{2,\dots,p\}$.

We construct the factorization blockwise in the Fourier domain in a way that is compatible with this pairing. Let $N=\{k\in\NN\mid 2\le k<\tfrac{p+2}{2}\}$. For each $k\in N$, set
$r_k=\rank(A_k)$. By the rank normal form (see, e.g., \cite[Sec.~0.4.6(f)]{HornJohnson2013}), there exist invertible matrices
$U_k,V_k\in\CC^{n\times n}$ such that
\begin{equation}\label{eq:rank-normal-form-idem}
U^{-1}_kA_kV^{-1}_k=\begin{bmatrix}I_{r_k}&0\\[2pt]0&0\end{bmatrix}=E_k.
\end{equation}
Equivalently, $A_k=U_kE_kV_k$, and $E_k^2=E_k$. We then define
\[
U_{p-k+2}=\overline{U_k}, \quad
V_{p-k+2}=\overline{V_k}, \quad
E_{p-k+2}=\overline{E_k},
\]
which yields compatible factorizations of the conjugate blocks.

For the real block $A_1$, we choose invertible $U_1,V_1\in\RR^{n\times n}$ so that
\eqref{eq:rank-normal-form-idem} holds. If $p$ is even, the same applies to $A_{\frac{p+2}{2}}$.

Now define block diagonal matrices
\[
U=\diag(U_1,\ldots,U_p),\qquad
V=\diag(V_1,\ldots,V_p),\qquad
E=\diag(E_1,\ldots,E_p),
\]
so that
\begin{equation}\label{eq:block-factorization-idem}
\diag(A_1,\ldots,A_p)=UEV,\qquad\text{and}\qquad E^2=E.
\end{equation}

We now define tensors $\cU,\cV,\cE$ by
\[
\bcirc(\cU)=(F_p^*\otimes I_n)U(F_p\otimes I_n), \quad
\bcirc(\cV)=(F_p^*\otimes I_n)V(F_p\otimes I_n), \quad
\bcirc(\cE)=(F_p^*\otimes I_n)E(F_p\otimes I_n).
\]

By construction, all the Fourier blocks above satisfy the conjugate-pairing conditions of Lemma~\ref{lem:realtensor}, and therefore $\cU$, $\cV$, and $\cE$ are real-valued tensors. Moreover, $\bcirc(\cU)$ and $\bcirc(\cV)$ are
invertible, so $\cU$ and $\cV$ are $t$-invertible.

Finally, combining \eqref{eq:fft-blockdiag-A-idem} and \eqref{eq:block-factorization-idem} yields
\[
\bcirc(\cA)=\bcirc(\cU)\,\bcirc(\cE)\,\bcirc(\cV).
\]
Applying $\bcirc^{-1}$ and Lemma~\ref{lem:bcirc-hom} gives $\cA=\cU*\cE*\cV$. Since $\bcirc(\cE)^2=\bcirc(\cE)$ and
$\bcirc$ is injective, we conclude $\cE*\cE=\cE$.
\end{proof}

The results of this section show that the real tensor setting requires more than a formal transfer of matrix factorizations to the Fourier domain. What is essential is that the matrix-level constructions be chosen compatibly across conjugate Fourier blocks. In the next section, we show that the same principle also governs the existence and construction of real generalized inverses.

\section{Generalized inverses of real tensors}\label{sec:g-inverses}

We now turn to generalized inverses in the $t$-product setting. The constructions developed in the previous section apply naturally to this context, as the defining identities of generalized inverses are expressed in terms of algebraic relations.

The main point in the real setting is that, although generalized inverses are constructed at the level of the Fourier blocks, the resulting tensor is real-valued only if the corresponding block data are chosen compatibly across conjugate pairs. This is precisely the mechanism encoded by the conjugate-pairing principle.

A ring element is said to be  \emph{unit regular} if it has an invertible von Neumann inverse. A ring is said to be \emph{unit regular} if every element is unit regular \cite{ChenBook}.

The following result is a direct consequence of the factorization established in Theorem \ref{thm:rank-factor-idempotent}.

\begin{theorem}\label{thm:unit-regular}
The ring $(\RR^{n\times n\times p},+,*)$ is unit regular.
\end{theorem}

\begin{proof}
Let $\cA\in\RR^{n\times n\times p}$. By Theorem \ref{thm:rank-factor-idempotent}, there exist $t$-invertible tensors
$\cU,\cV\in\RR^{n\times n\times p}$ and an idempotent tensor $\cE\in\RR^{n\times n\times p}$ such that
\[
\cA=\cU*\cE*\cV.
\]
Since $\cE^2=\cE$, we have $\cE=\cE*\cI*\cE$. Hence, $\cI\in\cE\{1\}$ and $\cE$ is unit regular. Let
\[
\cW=\cV^{-1}*\cU^{-1}.
\]
Then, $\cW$ is a unit and
\begin{align*}
\cA*\cW*\cA
&=(\cU*\cE*\cV)(\cV^{-1}*\cU^{-1})(\cU*\cE*\cV) \\
&=\cU*(\cE*\cI*\cE)*\cV \\
&=\cU*\cE*\cV \\
&=\cA.
\end{align*}
Therefore $\cA$ is unit regular. Since $\cA$ was arbitrary, the ring $(\RR^{n\times n\times p},+,*)$ is unit regular.
\end{proof}

%

Let $\cA\in\RR^{m\times n\times p}$ and let
\[
\cA=\cU*\cS*\cV^T
\]
be a real $t$-SVD of $\cA$, where $\cU\in\RR^{m\times m\times p}$ and $\cV\in\RR^{n\times n\times p}$ are orthogonal and
$\cS\in\RR^{m\times n\times p}$ is $f$-diagonal, chosen as described in the proof of
Theorem~\ref{thm:realtSVD} (see also Remark~\ref{rem:realtSVD}).

The tensor $\cS$ is defined by
\begin{equation}\label{eq:moorepenrose1}
\bcirc(\cS)
=
(F_p^*\otimes I_m)\diag(\Sigma_1,\ldots,\Sigma_p)(F_p\otimes I_n),
\end{equation}
where each $\Sigma_k$ is a diagonal matrix with nonnegative real entries. In particular,
the Fourier blocks $\Sigma_1,\ldots,\Sigma_p$ satisfy the conjugate-pairing conditions of
Lemma~\ref{lem:realtensor}. Moreover, each $\Sigma_k$ admits a real Moore--Penrose inverse
$\Sigma_k^\dagger$.

Let $N=\{k\in\NN\mid 2\le k<\tfrac{p+2}{2}\}$. By construction, for every $k\in N$ we have
$$\Sigma_{p-k+2}=\Sigma_k=\overline{\Sigma_k},$$
and therefore
$$\Sigma_{p-k+2}^\dagger=\Sigma_k^\dagger=\overline{\Sigma_k^\dagger}.$$

We now define the tensor $\cX$ by
\begin{equation}\label{eq:moorepenrose2}
\bcirc(\cX)=(F_p^*\otimes I_n)\diag(\Sigma_1^\dagger,\ldots,\Sigma_p^\dagger)(F_p\otimes I_m).
\end{equation}
Since the Fourier blocks in \eqref{eq:moorepenrose2} satisfy the conjugate-pairing
conditions of Lemma~\ref{lem:realtensor}, it follows that $\cX$ is real-valued.

Finally, by substituting \eqref{eq:moorepenrose1} and \eqref{eq:moorepenrose2} into the
Penrose equations and applying Lemma~\ref{lem:bcirc-hom}, the verification reduces
to the corresponding identities for the Fourier blocks. Since $\Sigma_k^\dagger$ is the
Moore--Penrose inverse of $\Sigma_k$ for each $k=1,\dots,p$, it follows immediately that
$\cX$ is the Moore--Penrose inverse of $\cS$. Hence, $\cX=\cS^\dagger$.

\begin{theorem}\label{thm:mp-from-tsvd}
Let \(\mathcal A \in \mathbb R^{m\times n\times p}\), and let
\[
\mathcal A = \mathcal U \ast \mathcal S \ast \mathcal V^T
\]
be a real $t$-SVD of \(\mathcal A\), where \(\mathcal S\) is defined as in \eqref{eq:moorepenrose1}. Then the
Moore--Penrose inverse of \(\mathcal A\) with respect to the $t$-product is given by
\[
\mathcal A^\dagger = \mathcal V \ast \mathcal S^\dagger \ast \mathcal U^T.
\]
In particular, \(\mathcal A^\dagger \in \mathbb R^{n\times m\times p}\).
\end{theorem}

\begin{proof}
Set $\cX=\cV*\cS^\dagger*\cU^T$. By associativity of the $t$-product and orthogonality of $\cU,\cV$, we compute
\[
\cA*\cX*\cA=\cU*(\cS*\cS^\dagger*\cS)*\cV^T=\cA,\qquad
\cX*\cA*\cX=\cV*(\cS^\dagger*\cS*\cS^\dagger)*\cU^T=\cX.
\]
Moreover, $\cA*\cX=\cU*(\cS*\cS^\dagger)*\cU^T$ and
$\cX*\cA=\cV*(\cS^\dagger*\cS)*\cV^T$. Since $\cS^\dagger$ is the Moore--Penrose inverse of $\cS$, both $\cS \ast \cS^\dagger$ and $\cS^\dagger \ast \cS$ are $t$-symmetric. Therefore, 
$(\cA*\cX)^T=\cA*\cX$ and $(\cX*\cA)^T=\cX*\cA$. Thus, $\cX$ satisfies the four Penrose equations, and hence
$\cX=\cA^\dagger$. Finally, $\cU,\cV$ and $\cS^\dagger$ are real-valued, so
$\cA^\dagger$ is also real-valued.
\end{proof}

\begin{remark}
The Moore--Penrose inverse also admits a direct Fourier-block description. Indeed, if
$$\bcirc(\cA)=(F_p^* \otimes I_m)\,\diag(A_1,\dots,A_p)\,(F_p \otimes I_n),$$
then
$$\bcirc(\mathcal A^\dagger)=
(F_p^* \otimes I_n)\,\diag(A_1^\dagger,\dots,A_p^\dagger)\,(F_p \otimes I_m),$$
where $A_k^\dagger$ denotes the matrix Moore--Penrose inverse of $A_k$, for
$k=1,\dots,p$.

If $\cA$ is real-valued, then Lemma \ref{lem:conj-pairing} gives $A_{p-k+2}=\overline{A_k}$, and hence
$A_{p-k+2}^\dagger = \overline{A_k^\dagger},$ so that Lemma \ref{lem:realtensor} implies that $\cA^\dagger$ is again real-valued.

The $t$-SVD formulation above is included because it provides an efficient computational tool for obtaining the Moore-Penrose inverse of a real-valued tensor.
\end{remark}

We now turn to the Drazin inverse. Its description is particularly natural in the Fourier domain, where the construction reduces blockwise to the corresponding matrix Drazin inverses, and real-valuedness is recovered through the conjugate-pairing principle.

We use the standard definition of the matrix Drazin inverse and its basic invariance properties (see, e.g., \cite[Sec.~7.2]{CampbellMeyer2009}).

\begin{theorem}\label{thm:drazin-fourier}
Let $\cA\in\RR^{n\times n\times p}$ with Fourier block diagonalization
\[
\bcirc(\cA)=(F_p^*\otimes I_n)\,\diag(A_1,\dots,A_p)\,(F_p\otimes I_n).
\]
Then the Drazin inverse of $\cA$, $\cA^D$, exists as a real tensor and satisfies
\[
\bcirc(\cA^D)=(F_p^*\otimes I_n)\,\diag(A_1^D,\dots,A_p^D)\,(F_p\otimes I_n).
\]
\end{theorem}

\begin{proof}

By Lemma~\ref{lem:bcirc-hom}, $\bcirc$ is a ring monomorphism, hence it preserves the Drazin equations. Moreover, $\bcirc(\cA)$ is similar to $\diag(A_1,\dots,A_p)$, and the Drazin inverse is similarity invariant. Furthermore, $\diag(A_1,\dots,A_p)^D=\diag(A_1^D,\dots,A_p^D)$, yielding the displayed formula.

If $\cA$ is real, Lemma~\ref{lem:conj-pairing} gives $A_1$ real and $A_{p-i+2}=\overline{A_i}$ for $i\ge2$. Since $A_1^D$ is real and $A_{p-i+2}^D=\overline{A_i^D}$ for $i\ge2$, applying Lemma~\ref{lem:realtensor}, we obtain that $\cA^D$ is real-valued.
\end{proof}

We finally consider the group inverse. Its existence and structure are again determined blockwise in the Fourier domain, with real-valuedness ensured by the conjugate-pairing principle.

\begin{theorem}\label{thm:real-group-inv}
Let $\cA\in\RR^{n\times n\times p}$ and
\[
\bcirc(\cA)=(F_p^*\otimes I_n)\,\diag(A_1,\dots,A_p)\,(F_p\otimes I_n).
\]
Then $\cA$ has a \emph{real-valued} group inverse $\cA^\#$ if and only if every $A_i$ is group invertible.
In that case,
\[
\bcirc(\cA^\#)=(F_p^*\otimes I_n)\,\diag(A_1^\#,\dots,A_p^\#)\,(F_p\otimes I_n).
\]
\end{theorem}

\begin{proof}
(\emph{Only if}.) If $\cA^\#$ exists, then $\bcirc(\cA)^\#$ exists by Lemma~\ref{lem:bcirc-hom}.
Since $\bcirc(\cA)$ is similar to $\diag(A_1,\dots,A_p)$, the latter is group invertible and
$\diag(A_1,\dots,A_p)^\#=\diag(A_1^\#,\dots,A_p^\#)$, hence each $A_i$ is group invertible.

(\emph{If}.) Assume each $A_i^\#$ exists and set $B_i=A_i^\#$. From Lemma~\ref{lem:conj-pairing},
$A_1$ is real and $A_{p-i+2}=\overline{A_i}$ for $i\ge2$, hence
\[
B_{p-i+2}=A_{p-i+2}^\#=(\overline{A_i})^\#=\overline{A_i^\#}=\overline{B_i},
\qquad i=2,\dots,p,
\]
(and $B_1$ is real; also $B_{\frac{p}{2}+1}$ is real if $p$ is even). Define
\[
B=\diag(B_1,\dots,B_p),\qquad 
\bcirc(\cB)=(F_p^*\otimes I_n)\,B\,(F_p\otimes I_n).
\]
Then $\cB$ is real-valued by Lemma~\ref{lem:realtensor}, and $\bcirc(\cB)=\bcirc(\cA)^\#$ by similarity invariance of
the group inverse. Applying $\bcirc^{-1}$ and Lemma~\ref{lem:bcirc-hom} gives $\cB=\cA^\#$, proving the claim.
\end{proof}

Thus, both the factorization theory and the generalized inverse theory for real tensors under the $t$-product are governed by the same Fourier-domain compatibility mechanism. This confirms that the passage from the complex to the real setting is not merely formal, but requires a coherent global treatment of conjugate Fourier data.

\section*{Acknowledgements}
\noindent This research was partially financed by Portuguese Funds through FCT (Fundação para a Ciência e a Tecnologia) within the Project UID/00013/2025. \\ \href{https://doi.org/10.54499/UID/00013/2025}{https://doi.org/10.54499/UID/00013/2025}

\section*{Data Availability}
No datasets were generated or analyzed during the current study. All results are theoretical and fully contained within the manuscript.

\section*{Conflict of interest}
The authors declare that they have no conflict of interest.

\end{document}